\documentclass[10pt,twoside, a4paper, english, reqno]{amsart}
\usepackage{graphicx, amsmath, varioref, amscd, amssymb,color, bm, stmaryrd, amsthm}
\usepackage{fancybox}
\usepackage[pdftex, colorlinks=true]{hyperref}

\usepackage{pdfsync}

\numberwithin{equation}{section}
\allowdisplaybreaks

\newtheorem{theorem}{Theorem}[section]
\newtheorem{lemma}[theorem]{Lemma}
\newtheorem{corollary}[theorem]{Corollary}

\newtheorem{proposition}[theorem]{Proposition}
\theoremstyle{definition}

\theoremstyle{remark}

\newcommand{\Div}{\operatorname{div}}

\newcommand{\Grad}{\nabla}
\newcommand{\vr}{\varrho}

\def\vphi{\varphi}

\newcommand{\weak}{\rightharpoonup}

\newcommand{\weakstar}{\overset{\star}\rightharpoonup}

\newcommand{\Set}[1]{\left\{#1\right\}}

\newcommand{\R}{\mathbb{R}}

\begin{document}

\title[Strong local alignment ]
	  {On strong local alignment \\ in the kinetic Cucker-Smale model} 

	\author[Karper]{Trygve K. Karper}\thanks{The work of T.K was supported by the Research Council of Norway through the project 205738}

	\address[Karper]{\newline
	Center for Scientific Computation and Mathematical Modeling, University of Maryland, College Park, MD 20742}
	\email[]{\href{karper@gmail.com}{karper@gmail.com}}
	\urladdr{\href{http://folk.uio.no/~trygvekk}{folk.uio.no/\~{}trygvekk}}

	\author[Mellet]{Antoine Mellet}\thanks{The work of A.M was supported by the National Science Foundation  under the Grant  DMS-0901340}

	\address[Mellet]{\newline
	Department of Mathematics, University of Maryland, College Park, MD 20742}
	\email[]{\href{mellet@math.umd.edu}{mellet@math.umd.edu}}
	\urladdr{\href{http://www.math.umd.edu/~mellet}{math.umd.edu/~{}mellet}}

	\author[Trivisa]{Konstantina Trivisa}\thanks{The work of K.T. was supported by the National Science Foundation   under the Grant  DMS-1109397}
	\address[Trivisa]{\newline
	Department of Mathematics, University of Maryland, College Park, MD 20742}
	\email[]{\href{trivisa@math.umd.edu}{trivisa@math.umd.edu}}
	\urladdr{\href{http://www.math.umd.edu/~trivisa}{math.umd.edu/~{}trivisa}}

\date{\today}



\maketitle

\begin{abstract}
In the recent papers \cite{KMT-2012-1, KMT-2012-2} the authors study the existence 
of weak solutions  and the hydrodynamic limit of  
kinetic flocking equations with 
strong local alignment. 
The introduction of a strong local alignment term to model flocking behavior
was  formally
motivated in these papers as a limiting case of an alignment term proposed by Motsch and Tadmor \cite{MT-2011}.
In this paper, we rigorously justify this limit, and show that the equation considered in \cite{KMT-2012-1, KMT-2012-2} is indeed a limit of the Motsch-Tadmor model when the radius of interaction goes to zero. 
The analysis involves velocity 
averaging lemmas and several $L^p$ estimates.
\end{abstract}

\section{Introduction}
In  \cite{MT-2011} Motsch and Tadmor identify an 
undesirable feature of the widely studied Cucker-Smale flocking model (cf.~\cite{CS-2007A, CS-2007B, HT-2008}):
In the Cucker-Smale model, the alignment 
of each individual is scaled with the total mass such that
the effect of alignment is almost negligible in sparsely populated
regions. 
To avoid this, they propose a new model in which
the alignment term is normalized with a local average density instead of 
the total mass.
Motivated by this work, the authors of the present paper proposed in \cite{KMT-2012-1} to combine 
the Cucker-Smale and Motsch-Tadmor models, letting the usual Cucker-Smale alignment term dominate the 
large scale dynamics and the Motsch-Tadmor term
the small scale dynamics. This remedies the aforementioned deficiency 
while maintaining the large scale dynamics of the Cucker-Smale model.
At the mesoscopic level, the proposed model takes the following form
\begin{equation}\label{eq:1}
	\begin{split}
		f_t + \Div_x(vf) - \Div_v(f\Grad_x \Psi) + \Div_v (f F[f]) + \Div_v(f L^r[f]) = 0.
	\end{split}
\end{equation}
Here, the unknown is the distribution function $f:= f(t,x,v)$.
The first alignment term $F[\cdot]$ is the standard Cucker-Smale alignment term given by
\begin{equation}\label{eq:LCS}
	F[f(x,v)] = \int_{\R^{2d}} \Phi(x-y)f(y,w)(w-v)~dwdy,
\end{equation}
where $\Phi(x)$ is the influence function (e.g $\Phi(x) = 1/(1+|x|^2)$).
The second alignment term $L^r[\cdot]$ is the Motsch-Tadmor alignment term  given by (see \cite{MT-2011}):
\begin{equation}\label{eq:L}
	L^r[f(x,v)] = \frac{\int_{\R^{2d}} K^r(x-y)f(y,w)(w-v)~dwdy}{\int_{\R^{2d}} K^r(x-y)f(y,w)~dwdy}.
\end{equation}
where the index $r$ denotes the radius of influence of $K^r$ (see \eqref{eq:Rr} below for the definition of $K^r$). 
Finally, the function $\Psi(x)$ 
is a given confinement potential introduced to avoid mass vanishing to infinity (it satisfies $\lim_{|x|\to\infty} \Psi(x)=\infty$). 
Note that this term is not necessary if we assume, for instance, that the initial distribution $f(0,x,v)$ has compact support in $x$ and $v$ (since \eqref{eq:1} will propagate this property).

The only difference between \eqref{eq:LCS} and \eqref{eq:L} is the renormalization by the local average density $\int_{\R^{2d}} K^r(x-y)f(y,w)~dwdy$. 
We can also write $L^r$ as follows: 
$$ L^r(f)=\tilde u^r- v$$
where 
$$ \tilde u^r(x) = \frac{\int_{\R^{2d}} K^r(x-y) w f(y,w)~dwdy}{\int_{\R^{2d}} K^r(x-y)f(y,w)~dwdy}.$$
In this form, it is obvious that the strength of the alignment force is now independent of the total mass, which was the original intend of \cite{MT-2011}. 
Another effect of this renormalization is to break the symmetry of the alignment. As a consequence, \eqref{eq:1} does not conserve momentum nor energy, and the derivation of an energy bound will be one of the main difficulty in the analysis of \eqref{eq:1}.

The purpose of this paper is to study the limit $r \rightarrow 0 $ in Equation \eqref{eq:1} 
when the function $K^r$ converges to the Dirac distribution $\delta_0$.
In other words, we study the limit of  \eqref{eq:1} when the Motsch-Tadmor term $\Div_v (fL^r[f])$ 
becomes a local (in space) alignment term.
For the sake of simplicity, we assume that $K^r$ has the form
\begin{equation}\label{eq:Rr}
	K^r(x) = r^{-d}K\left(\frac{x}{r}\right), 
	\end{equation}
where $K$ is a given function satisfying
\begin{equation}\label{eq:R} 
K\in C_c(\R^d), \quad K(0)>0,\quad \int_{\R^d}K^r(x)~dx = 1.
\end{equation}
When $r\to0$, we then formally expect to have
$$ \tilde u^r (x)\overset{r \rightarrow 0}{\longrightarrow} u(x)= \frac{\int_{\R^{d}}  w f(x,w)~dw}{\int_{\R^{d}} f(x,w)~dw}
$$
and so
\begin{equation}\label{eq:Mconv}
	\begin{split}
			L^r[f(x,v)] 
			& \overset{r \rightarrow 0}{\longrightarrow} \frac{\int_{\R^d}f(x,w)(w-v)~dw}{\int_{\R^d}f(x,w)~dw} := u-v.
	\end{split}
\end{equation}
Passing to the limit in \eqref{eq:1}, we thus obtain the equation
\begin{equation}\label{eq:2}
	f_t + \Div_x(vf) - \Div_v(f\Grad_x \Psi) + \Div_v (f F[f]) + \Div_v(f(u-v)) = 0,
\end{equation}
which is studied in \cite{KMT-2012-1, KMT-2012-2}.
The new local alignment term can also be seen as a local friction term centered at $u$. 

The purpose of this paper is to rigorously justify this limit $r\to 0$.
More precisely, we will prove  the following theorem:
\begin{theorem}\label{thm:main}
Let $0 \leq f_0 \in L^1(\R^{2d})\cap L^\infty(\R^{2d})$ be given and 
$T$ be a finite final time. For each $r>0$, let $f^r(t,x,v)$ be a weak solution
of \eqref{eq:1} in the sense that
\begin{align}\label{eq:weak}
			&\int_{\R^{2d+1}} - f^r \phi_t - vf^r\Grad_x \phi +f^r\Grad_x \Psi \Grad_x \phi - f^rF[f^r]\Grad_v \phi ~dvdxdt\\
			&- \int_{\R^{2d+1}}f^rL^r[f^r]\Grad_v \phi~dvdxdt
			= \int_{\R^{2d}} f_0 \phi(0,\cdot)~dvdx, \quad \forall \phi \in C_c^\infty([0,T)\times \R^{2d}), \nonumber
\end{align}
where $L^r$ is given by \eqref{eq:L} and $K^r$ is given by \eqref{eq:Rr}.
Then, as $r\rightarrow 0$,
\begin{equation*}
	\begin{split}
			f^r &\weakstar f \quad \text{in $L^\infty(0,T;L^\infty(\R^{2d})\cap L^1(\R^{2d}))$},\\ 
			f^rL^r[f^r] &\weak f(u-v)\quad \text{in $L^q((0,T)\times \R^{2d})$}, ~ q< \frac{d+2}{d+1},
	\end{split}
\end{equation*}
with $u(t,x)$ such that $j=\rho u$ (see \eqref{eq:uu1} for a precise definition). 
Furthermore, the limit $f(t,x,v)$ is a weak solution of \eqref{eq:2} in the sense that
\begin{align}\label{eq:weak2}
			&\int_{\R^{2d+1}} - f \phi_t - vf\Grad_x \phi +f\Grad_x \Psi \Grad_x \phi - fF[f]\Grad_v \phi ~dvdxdt\\
			&- \int_{\R^{2d+1}}f(u-v)\Grad_v \phi~dvdxdt 
			= \int_{\R^{2d}} f^0 \phi(0,\cdot)~dvdx, \quad \forall \phi \in C_c^\infty([0,T)\times \R^{2d}). \nonumber
\end{align}

\end{theorem}

\section{Preliminary material}
In this section we have gathered some results 
that will be needed to prove Theorem \ref{thm:main}.
We begin by introducing some convenient notations. 
We denote the first and second moments of $f$,
and their $K^r$ weighted counterparts, as follows:
\begin{equation*}
	\begin{split}
		\vr(t,x) &= \int_{\R^{2d}}f(y,w)~dwdy, \quad \qquad \widetilde \vr^r(t,x) = \int_{\R^{2d}}K^r(x-y)f(y,w)~dwdy, \\
		j(t,x) &= \int_{\R^{2d}}f(y,w)w~dydw, \quad \quad \phantom{~} \widetilde j^r(t,x) = \int_{\R^{2d}}K^r(x-y)f(y,w)w~dwdy.
	\end{split}
\end{equation*}
We also define the corresponding velocities
\begin{equation*}
	u(t,x) = \frac{j(t,x)}{\vr(t,x)}, \qquad \widetilde u^r(t,x) = \frac{\widetilde j^r(t,x)}{\widetilde \vr^r(t,x)}.
\end{equation*}
Note that the definition of $u$ (and $\widetilde u$) is ambiguous if $\rho$ (resp. $\widetilde \vr$) vanishes. We thus define $u$ pointwise by
\begin{equation}\label{eq:uu1} 
u(x,t)=\left\{ \begin{array}{ll}\displaystyle \frac{j(x,t)}{\rho(x,t)} & \mbox{ if }  \rho(x,t) \neq 0 \\[8pt] 0 &  \mbox{ if }  \rho(x,t) = 0 \end{array}\right. .
\end{equation}
Since we have
$$j \leq  \left(\int |v|^2 f(x,v,t)\,  dv\right)^{1/2} \rho^{1/2},$$
the bound on the energy of $f$ will imply that $j=0$ whenever $\rho=0$ and so (\ref{eq:uu1}) implies in particular $j=\rho u$.

With the above notation, we have $L^r[f] = \widetilde u^r - v$, 
and \eqref{eq:1} can be written as
\begin{equation}\label{eq:22}
	\begin{split}
		f_t + \Div_x(fv) - \Div_v(f \Grad_x \Psi) + \Div_v(fF[f]) + \Div_v(f(\widetilde u^r - v)) = 0.
	\end{split}
\end{equation}

The following proposition states that \eqref{eq:22} is well-posed in the sense 
of weak solutions (see \cite{KMT-2012-1} for the proof).
\begin{proposition}\label{pro:}
Assume that $0 \leq f_0 \in [L^\infty\cap L^1](\R^{2d})$ and $T < +\infty$ are given. 
Then, for any $r> 0$, \eqref{eq:22} admits a weak solution $0 \leq f \in C(0,T;L^1(\R^{2d}))$.
Moreover,  $f$ satisfies
\begin{equation}\label{eq:fLp}
	\|f\|_{L^\infty(0,T;L^p(\R^{2d}))} \leq e^{\frac{CT}{p'}}\|f_0\|_{L^p(\R^{2d})},
\end{equation}
\begin{equation}	 \label{eq:fEr}
 \mathcal{E}(t):= \int_{\R^{2d}} \left(|v|^2 +\Psi(x)\right) f(t,x,v) \, dv\, dx \leq Ce^{CT}\mathcal{E}(0),
\end{equation}
where the constant $C$ might depend on $r$.
\end{proposition}

To conclude this section, we recall the following classical lemma, which will be used to derive
$L^p$ integrability of $\vr$ and $j$ (see \cite{KMT-2012-1} for the proof):
\begin{lemma}\label{lem:jn}
Assume that $f$ satisfies
$$ \|f \|_{L^{\infty}([0,T]\times \R^{2d})} \leq M, \quad \mbox{ and } \quad \int_{\R^{2d}} |v|^2 f \, dv dx \leq M .$$
Then there exists a constant $C=C(M) $ such that
\begin{equation}\label{eq:jnest}
\begin{array}{l}
\|\rho\|_{L^\infty(0,T;L^{p}(\R^d))} \leq C ,\quad\mbox{ for every $p\in[1,\frac{d+2}{d}),$}\\ [5pt]
\| j\|_{L^\infty(0,T;L^{p}(\R^d))} \leq C,\quad\mbox{ for every $p\in[1,\frac{d+2}{d+1}),$}
\end{array}
\end{equation}
where $\rho=\int f\, dv$ and $j=\int vf\,dv $.
\end{lemma}
%

\subsection{The Velocity Averaging Lemma}
When passing to the limit in \eqref{eq:22}, the main obstacle 
is to obtain compactness of the product $f\widetilde u^r$. 
The instrument we will use to obtain this is the 
celebrated velocity averaging lemma. We 
will use the following version due to Perthame \& Souganidis \cite{PS-1998}.

\begin{proposition}\label{pro:velocity}
Let $\{f^n\}_n$ be bounded in $L_\text{loc}^p(\R^{2d+1})$ with $1 < p < \infty$, 
and $\{G^n\}_n$ be bounded in $L_\text{loc}^p(\R^{2d+1})$. If 
$f^n$ and $G^n$ satisfy
$$
f^n_t + v\cdot \Grad_x f^n = \Grad_v^k G^n, \qquad f^n|_{t=0} = f^0 \in L^p(\R^{2d}),
$$
for some multi-index $k$ and $\vphi \in C^{|k|}_\text{c}(\R^{2d})$,
then $\{\vr_\vphi^n\}$ is relatively compact in $L^p_\text{loc}(\R^{d+1})$.
\end{proposition}

The previous proposition cannot be directly applied 
to obtain the needed compactness. In fact, 
we will rely on the following lemma which can be 
seen as a corollary of the previous proposition. 
The proof can be found in \cite{KMT-2012-1}.

\begin{lemma}\label{cor:velocity}
Let $\{f^n\}_n$ and $\{G^n\}_n$ be as in Proposition \ref{pro:velocity} 
and assume that
\begin{equation*}
f^n \mbox{ is bounded in } L^\infty(\R^{2d+1}),
\end{equation*}
\begin{equation*}
 (|v|^2+\Psi)f^n \mbox{ is bounded in } L^\infty(0,T;L^1(\R^{2d+1})).
\end{equation*}
Then, for any $\vphi(v)$ such that $|\vphi(v)| \leq c|v|$ and $q < \frac{d+2}{d+1}$, the sequence
\begin{equation}\label{eq:sec}
	\Set{\int_{\R^d}f^n \vphi(v)~dv}_n,
\end{equation}
is relatively compact in $L^q((0,T)\times \R^d)$.
\end{lemma}

\subsection{An important technical lemma}
In view of Lemma \ref{lem:jn} and \ref{cor:velocity}, it is clear that in order to get 
convergence results for $f^r$ and its moments, we will need to obtain some estimate on $f^r$ that are uniform with respect to $r$. 
The main difficulty will be to show that the energy estimate (\ref{eq:fEr}) holds with constants independent on $r$ (which does  not obviously follows  from the result of \cite{KMT-2012-1}).
For this we will make use of the following technical lemma, which can be found in \cite{KMT-2012-1}
(the proof is given below for completeness):

\begin{lemma}\label{lem:MT}
Assume that there exists $0<R_1<R_2<\infty$ such that 
\begin{equation}\label{eq:condphi} 
K(x) >0 \mbox{ for } |x|\leq R_1\, , \qquad  K(x)=0  \quad \mbox{ for } |x|\geq R_2.
\end{equation}
There exists a constant 
\begin{equation}\label{eq:CC}
C\sim \frac{\sup_{B_{R_2}} K}{\inf_{B_{R_1}} K}  \left( \frac{R_2}{R_1}\right)^d 
\end{equation}
such that
$$ 
\int_{\R^{d}} K(x-y)   \frac{\rho(x)}{\int_{\R^{d}} K(x-z) \rho(z)\, dz} \, dx \leq C, \quad \forall y\in \R^d,
$$
for all nonnegative functions $\rho\in L^1(\R^d)$.
\end{lemma}
The most important part of this lemma is the formula \eqref{eq:CC}, which implies that if we replace the function $K$ with  $\alpha K(\beta x)$, for any $\alpha>0$ and $\beta>0$, then
the same results holds with the same constant.

We deduce:
\begin{corollary}\label{cor:Rr}
Assume that $K^r$ is given by \eqref{eq:Rr} where $R$ satisfies \eqref{eq:R}. Then, there exists a constant $C$ independent of $r$ such that
$$ 
\int_{\R^{d}}  K^r(x-y)   \frac{\rho(x)}{\int_{\R^{d}} K^r(|x-z|) \rho(z)\, dz} \, dx \leq C, \quad \forall y\in \R^d
$$
for all nonnegative functions $\rho\in L^1(\R^d)$.
\end{corollary}

\begin{proof}[Proof of Lemma \ref{lem:MT}]
We recall that $\tilde \rho(x) = \int_{\R^{d}} K(x-z) \rho(z)\, dz $ and we note that
$$\int_{\R^{d}} K(x-y)   \frac{\rho(x)}{\tilde \rho(x)} \, dx \leq  (\sup K) \int_{ B_{R_2}(y)}  \frac{\rho(x)}{\tilde \rho(x)} \, dx.
$$
Next, we cover $B_{R_2}(y)$ with balls of radius $R_1/2$: We have 
$$ B_{R_2}(y)\subset \bigcup_{i=1}^N B_{R_1/2}(x_i)$$
with $N\sim(R_2/R_1)^d$.
We can thus write
$$\int_{\R^{d}} K(x-y)   \frac{\rho(x)}{\tilde \rho(x)} \, dx \leq 
 (\sup K)\sum_{i=1}^N \int_{B_{R_1/2}(x_i)}  \frac{\rho(x)}{\tilde \rho(x)} \, dx.
$$
Moreover, clearly, 
\begin{eqnarray*}
\tilde \rho(x) = \int_{\R^{d}} K(x-z)   \rho(z) \, dz\geq \int_{B_{R_1/2}(x_i)}  K(x-z)   \rho(z) \, dz .
\end{eqnarray*}

By combining the two previous inequalities, we see that
$$\int_{\R^{d}} K(x-y)   \frac{\rho(x)}{\tilde \rho(x)} \, dx \leq 
 (\sup K)\sum_{i=1}^N \int_{B_{R_1/2}(x_i)}  \frac{\rho(x)}{\int_{B_{R_1/2}(x_i)}  K(x-z)   \rho(z) \, dz} \, dx.
$$
Now, using the fact that when  $ x,z\in B_{R_1/2}(x_i)$ we have $|x-z|\leq R_1$, we deduce
\begin{align*}
\int_{\R^{d}} K(x-y)   \frac{\rho(x)}{\tilde \rho(x)} \, dx& \leq 
 \frac{\sup K}{\inf_{B_{R_1}(0)} K} \sum_{i=1}^N \int_{B_{R_1/2}(x_i)}  \frac{\rho(x)}{\int_{B_{R_1/2}(x_i)}   \rho(z) \, dz} \, dx\\
 & \leq   \frac{\sup K}{\inf_{B_{R_1}(0)} K}  N \leq C  \frac{\sup K}{\inf_{B_{R_1}(0)} K}  \left( \frac{R_2}{R_1}\right)^d
\end{align*} 
and the proof is complete.
\end{proof}

\subsection{A priori estimate}
We can now conclude this preliminary section by proving that $f^r$ satisfies some a priori estimates uniformly with respect to $r$.
We recall that the energy functional is defined
\begin{equation}
	\mathcal{E}(t) = \int_{\R^{2d}}\left(\frac{|v|^2}{2} + \Psi(x)\right) f(t,x,v)~dvdx.
\end{equation}
We then prove:
\begin{proposition}[Energy bound]\label{pro:energy}
Let $0 \leq f_0 \in L^1(\R^{2d})\cap L^\infty(\R^{2d})$ be given, let 
$T$ be a finite final time, and let $f$ be the corresponding weak solution 
of \eqref{eq:1}. There is a constant $C > 0$ independent of $r>0$ such that  such that
\begin{equation}\label{eq:fLpunif}
	\|f\|_{L^\infty(0,T;L^p(\R^{2d}))} \leq e^{\frac{CT}{p'}}\|f_0\|_{L^p(\R^{2d})},
\end{equation}
and 
\begin{equation}\label{eq:fE}
	\begin{split}
		&\sup_{t\in (0,T)}\mathcal{E}(t) + \frac{1}{2}\int_{\R^{2d}}f|\widetilde u - v|^2~dvdx \\
		&\qquad \qquad + \frac{1}{2}\int_{\R^{4d}}\Phi(x-y)f(x,v)f(y,w)|w-v|^2~dwdydvdx \leq C(T)\mathcal{E}(0).
	\end{split}
\end{equation}
\end{proposition}

The proof of Proposition \ref{pro:energy} relies 
on two auxiliary results (Lemmas \ref{lem:Lp} and \ref{lem:Lbound} below) which we will prove prior 
to proving the proposition. 
We begin with the $L^p$ estimate \eqref{eq:fLpunif}:
\begin{lemma}\label{lem:Lp}
Let $f$ be a weak solution of \eqref{eq:1}. There is a constant $C$, 
independent of $r$, such that 
\begin{equation}
	\sup_{t \in (0,T)} \|f\|_{L^p(\R^{2d})} \leq \|f_0\|_{L^p(\R^{2d})}e^{CT}.
\end{equation}	
	
\end{lemma}

\begin{proof}
Let $B(f)$ be a continuous function and let $b(f) = fB'(f)- B(f)$. 
By multiplying \eqref{eq:1} with $B'(f)$ and integrating, we obtain
	\begin{equation}\label{eq:en1}
		\begin{split}
			\frac{d}{dt}\int_{\R^{2d}} B(f)~dvdx
			& = \int_{\R^{2d}} v\Grad_x b(f)~dvdx  \\
			&\quad + \int_{\R^{2d}} (F(f) + L^r(f)-\Grad_x \Psi)\Grad_v b(f)~dvdx \\
			&= -\int_{\R^{2d}}b(f)\left(\Div_v F(f) + \Div_v L^r(f) \right)~dvdx.
		\end{split}
	\end{equation}
Next, using the definition of the alignment terms, we see that
\begin{equation*}
	\begin{split}
		\Div_v F[f] &= -d\int_{\R^{2d}}\Phi(x-y)f(y,w)~dwdy, \\
		\Div_v L^r[f] &= -d \frac{\int_{\R^{2d}}K^r(x-y)f(y,w)~dwdy}{\int_{R^{2d}}K^r(x-y)f(y,w)~dwdy} = -d.
	\end{split}
\end{equation*}
Setting these identities in \eqref{eq:en1}, we find that
\begin{equation*}
	\begin{split}
		\frac{d}{dt}\int_{\R^{2d}} B(f)~dvdx 
		&= \int_{\R^{2d}}b(f)\left(d+ d\int_{\R^{2d}}\Phi(x-y)f(y,w)~dwdy\right)~dvdx \\
		&\leq \int_{\R^{2d}}b(f)(d  + dM\| \Phi \|_{L^\infty(\R^{2d})})~dvdx,
	\end{split}
\end{equation*}
where $M$ is the total mass.
Next, we let $B(f) = f^p$ such that $b(f)= (p-1)f^p$. An application of 
the Gronwall inequality then provides the bound
\begin{equation*}
	\begin{split}
		\sup_{t\in (0,T)}\|f\|_{L^p(\R^{2d})} \leq \|f_0\|_{L^p(\R^{2d)}}e^{\frac{p-1}{p}CT},
	\end{split}
\end{equation*}
which is what we set out to prove.

\end{proof}

The main difficulty in proving the energy estimate \eqref{eq:fE} (even for $r>0$) is to control the non-symmetric Motsch-Tadmor alignment term. 
This is the goal of the following Lemma, which relies on Lemma \ref{lem:MT}:
\begin{lemma}\label{lem:Lbound}
There is a constant $C$, independent of $r$, such that
	\begin{equation}
		\int_{\R^{2d}}fvL^r[f]~dvdx \leq C\mathcal{E}(t)-\frac{1}{2}\int_{\R^{2d}}f|\widetilde u^r - v|^2~dvdx  .
	\end{equation}
\end{lemma}

\begin{proof}
By definition of $L^r$, we have that
\begin{equation}
	\begin{split}
		L^r[f] &= \frac{1}{\widetilde \vr(x)}\int_{\R^{2d}}K^r(x-y)f(y,w)(w-v)~dwdy \\
		&= \frac{1}{\widetilde \vr(x)}\int_{\R^d}K^r(x-y)(j(y) - \vr(y)v)~dy 
		= \frac{\widetilde j}{\widetilde\vr} - v := \widetilde u - v.
	\end{split}
\end{equation}
By adding and subtracting, we obtain
\begin{equation}\label{eq:setback}
	\begin{split} 
		\int_{\R^{2d}}fvL^r[f]~dvdx 
		&=\int_{\R^{2d}}f(\widetilde u - v)v~dvdx   \\
		&= -\frac{1}{2}\int_{\R^{2d}}f(\widetilde  u - v)^2~dvdx + \frac{1}{2}\int_{\R^{2d}}f \widetilde u^2 - fv^2~dvdx \\
		&\leq -\frac{1}{2}\int_{\R^{2d}}f(\widetilde  u - v)^2~dvdx + \frac{1}{2}\int_{\R^{d}}\varrho \widetilde u^2~dx.
	\end{split}
\end{equation}
From the H\"older inequality, we have that
\begin{equation*}
	\begin{split}
		\widetilde \vr \widetilde u := \int_{\R^{2d}}K^r(x-y)f(y,v)v~dvdy 
		&\leq \widetilde \vr^\frac{1}{2}\left(\int_{\R^{2d}}K^r(x-y)f(y,v)v^2~dvdy \right)^\frac{1}{2}.
	\end{split}
\end{equation*}
Hence, the following inequality holds
\begin{equation*}
	\begin{split}
	\widetilde \vr \widetilde u^2
		&\leq \int_{\R^{2d}}K^r(x-y)f(y,v)v^2~dvdy,
	\end{split}
\end{equation*}
from which we  deduce
\begin{equation}\label{eq:product}
	\begin{split}
		\int_{\R^{d}}\varrho \widetilde u^2~dx &\leq \int_{\R^{3d}}K^r(x-y)\frac{\vr(x)}{\widetilde \vr(x)}f(y,v)v^2~dydvdx \\
		&\leq \sup_y\left(\int_{\R}K^r(x-y)\frac{\vr(x)}{\widetilde \vr(x)}~dx\right)\mathcal{E}(t) \leq C\mathcal{E}(t),
	\end{split}
\end{equation}
where the last inequality follows from Lemma \ref{lem:MT}. Inserting \eqref{eq:product} 
in \eqref{eq:setback} concludes the proof.

\end{proof}

We have now gathered all the ingredients we need to prove Proposition \ref{pro:energy}.

\subsection*{Proof of Proposition \ref{pro:energy}}
Only \eqref{eq:fE} remains to be proved.
By direct calculation, 
\begin{align}\label{eq:hvafaen}
		\frac{d}{dt}\int_{\R^{2d}}f\Psi + f\frac{|v|^2}{2}~dvdx 
		&= \int_{\R^{2d}}f_t \Psi + f_t\frac{|v|^2}{2} ~dvdx \\
		&= \int_{\R^{2d}}vf \Grad_x \Psi - vf\Grad_x \Psi + fv F[f] + fvL^r[f] ~dvdx\nonumber.
\end{align}
Using the symmetry of $K$, we write
\begin{equation*}
	\begin{split}
		\int_{\R^{2d}}fv F[f]~dvdx 
		&=\int_{\R^{4d}} \Phi(x-y)f(x,v)f(y,w)(w-v)v~dwdydvdx \\
		&= \int_{\R^{4d}} \Phi(x-y)f(x,v)f(y,w)(v-w)w~dwdydvdx \\
		&= -\frac{1}{2}\int_{\R^{4d}}\Phi(x-y)f(x,v)f(y,w)|w-v|^2~dwdydvdx.
	\end{split}
\end{equation*}
Then, we conclude the proof by applying this identity and Lemma \ref{lem:Lbound} to \eqref{eq:hvafaen}.

\qed

\section{Convergence and proof of Theorem \ref{thm:main}}
Equipped with the bounds of the previous section, 
we are ready to send $r \rightarrow 0$ in 
\eqref{eq:1} and thereby proving Theorem \ref{thm:main}.
For this purpose, we let $\{r^n\}_n$ be a sequence 
of positive numbers such that $r^n\rightarrow 0$ 
as $n \rightarrow \infty$ and consider 
the corresponding solutions $f^n$ of
\begin{equation}\label{eq:approx}
	f^n_t + \Div_x(vf^n) - \Div_v(f^n\Grad_x \Psi) + \Div_v(f^nF[f^n]) + \Div_v(f^n(\widetilde u^n - v)) = 0,
\end{equation}
where we recall the notation
\begin{equation*}
	\widetilde u^n = \frac{\widetilde j^n}{\widetilde \vr^n} := \frac{\int_{\R^{2d}}K^{r^n}(x-y)f^n(y,w)w~dwdy }{\int_{\R^{2d}}K^{r^n}(x-y)f^n(y,w)~dydw}.
\end{equation*}

Our starting point is that Lemma \ref{lem:Lp}, Proposition \ref{pro:energy}, together 
with Lemma \ref{lem:jn}, asserts the existence of a function
$0 \leq f \in C(0,T;L^1(\R^{2d}))\cap L^\infty(0,T;L^\infty(\R^{2d}))$, 
such that, as $n \rightarrow \infty$,
\begin{equation}\label{eq:weak}
	\begin{split}
	f^n &\weakstar f \quad \text{in $L^\infty(0,T;L^\infty(\R^{2d})\cap L^1(\R^{2d}))$},\\ 
	\vr^n &\weakstar \vr\quad \text{in $L^\infty((0,T);L^p(\R^d))$, $\quad$ for every $p \in \left[1,\frac{d+2}{d}\right)$}, \\
	j^n &\weakstar j \quad \text{in $L^\infty((0,T);L^p(\R^d))$, $\quad$ for every $p \in \left[1,\frac{d+2}{d+1}\right)$}.
	\end{split}
\end{equation}
Moreover, the velocity averaging Lemma \ref{cor:velocity} is applicable. 
By setting $\vphi(v) = 1$ and $\vphi(v)= v$ in Lemma \ref{cor:velocity} we obtain respectively
\begin{equation}\label{eq:strong}
	\begin{split}
			\vr^n &\rightarrow  \vr\quad \text{in $L^q((0,T)\times \R^d)$, $\quad$ for every $q < \frac{d+2}{d+1}$}, \\
			j^n &\rightarrow  j\quad \text{in $L^q((0,T)\times \R^d)$, $\quad$ for every $q < \frac{d+2}{d+1}$},
	\end{split}
\end{equation}
along some subsequence as $n \rightarrow \infty$.
Furthermore, we can prove:
\begin{lemma}\label{lem:tilde}
Given the convergences \eqref{eq:weak} - \eqref{eq:strong}, we have
\begin{equation}
	\widetilde j^n \rightarrow j, \quad \widetilde \vr^n \rightarrow \vr, \quad \text{in $~L^q((0,T)\times \R^d)$},
\end{equation}
where the convergence takes place along the same subsequence as in \eqref{eq:strong}. 
\end{lemma}
\begin{proof}
We commence by recalling the following classical results 
concerning mollifiers like $K^n=K^{r^n}$:
For any $\epsilon > 0$, there is a $m$ such that
\begin{equation*}
	\|K^n\star \vr - \vr\|_{L^q(\R^d)} < \epsilon, \quad \forall~ n \geq m.
\end{equation*}
Now, consider a subsequence $n^k$, where $n^k \geq m$, 
along which $\vr^n \rightarrow \vr$. By adding and subtracting,
we obtain
\begin{equation*}
	\begin{split}
		\|\widetilde \vr^n - \vr\|_{L^q(\R^d)} 
		&\leq \left\|\widetilde \vr^n - K^n\star \vr\right\|_{L^q(\R^d)} + \|K^n \star \vr - \vr\|_{L^q(\R^d)} \\
		&= \left\|K^n\star (\vr^n - \vr)\right\|_{L^q(\R^d)} + \|K^n \star \vr - \vr\|_{L^q(\R^d)} \\
		&\leq \|\vr^n - \vr\|_{L^q(\R^d)} + \epsilon = 2\epsilon,
	\end{split}
\end{equation*}
for any  $q < \frac{d+2}{d+1}$. The same argument can be applied to prove compactness of $\widetilde j^n$, which concludes the proof.

\end{proof}

\begin{lemma}\label{lem:product}
From the convergences \eqref{eq:weak} - \eqref{eq:strong}, it follows that
	\begin{equation*}
		f^n \widetilde u^n \weakstar fu
		 ~ \text{ in $L^\infty((0,T);L^p(\R^d))$$~$ for every $~p \in \left[1,\frac{d+2}{d+1}\right)$}.
	\end{equation*}
\end{lemma}
\begin{proof}
For the sake of clarity, let us introduce the notation
\begin{equation*}
	\vr^n_\vphi = \int_{\R^d}f^n\vphi(v)~dv, \qquad \widetilde m^n_\vphi = \widetilde u^n \vr^n_\vphi.
\end{equation*}
For any smooth function $\psi(x,v):=\phi(x)\vphi(v)$, we write
\begin{equation}\label{eq:topass}
	\begin{split}
		\int_{\R^{2d}}f^n \widetilde u^n \psi~dvdx&=
		\int_{\R^d} \widetilde u^n \phi(x)\left(\int_{\R^d}f^n \vphi(v)~dv\right)~dx \\
		&= \int_{\R^d}\widetilde u^n \vr^n_\vphi \phi~dx
		=\int_{\R^d}\widetilde m^n_\vphi \phi~dx.
	\end{split}
\end{equation}

Now, using the H\"older inequality, we find that
\begin{equation}\label{eq:mphi}
	\begin{split}
		\|\widetilde m^n_\vphi\|_{L^q(\R^d)} &\leq \|\vphi\|_{L^\infty(\R^d)}\|\vr^n\|^\frac{1}{2}_{L^{\frac{q}{2-q}}(\R^d)}\|(\vr^n)^\frac{1}{2}\widetilde u^n\|_{L^2(\R^d)} \\
		&\leq C\|\vr^n\|_{L^p(\R^d)}^\frac{1}{2}\|(\vr^n)^\frac{1}{2}\widetilde u^n\|_{L^2(\R^d)},
	\end{split}
\end{equation}
which is bounded by \eqref{eq:product} and Lemma \ref{lem:jn} provided 
\begin{equation*}
	p < \frac{d+2}{d} \quad \Rightarrow\quad q < \frac{d+2}{d+1}.
\end{equation*}
Hence, there exists a function $m \in L^\infty((0,T);L^q(\R^d))$ 
and a subsequence such that
\begin{equation*}
	\widetilde m^n_\vphi \weakstar m ~ \text{ in $L^\infty((0,T);L^p(\R^d))$, $\quad$ for every $p \in \left[1,\frac{d+2}{d+1}\right)$},
\end{equation*}
and it only remains to prove that
\begin{equation*}
	m = u\vr_\vphi, \quad \text{where $u$ is such that}\quad j = \vr u.
\end{equation*}

Let us first verify the existence of such a function $u$. 
Consider the set 
$$
A_R = \Set{(t,x) \in B_R(0)\times (0,T); \vr(t,x) = 0},
$$
where $B_R(0)$ is the ball of radius $R$ centered at $0$. By direct calculation,
\begin{equation*}
	\begin{split}
		\int_{A_R}|j_n|~dxdt 
		&\leq \left(\int_{A_R}\vr^n|u^n|^2~dxdt\right)^\frac{1}{2}\left(\int_{A^R}\vr^n~dxdt\right) \\
		&\leq CT\left(\int_{A_R}\vr^n~dxdt\right) \overset{n \rightarrow \infty}{\longrightarrow} 0,
	\end{split}
\end{equation*}
and hence we have that $j = 0$ a.e in $A^R$. If we define the function $u$ as
\begin{equation}
	u(t,x) = 
	\begin{cases}
		\frac{j(t,x)}{\vr(t,x)}, & \text{if }\vr(t,x) \neq 0, \\
		0, & \text{if }\vr(t,x) = 0,
	\end{cases}
\end{equation}
we have that $j = \vr u$ and it remains to prove that $m = \vr_\psi u$.
To this aim, we first observe that we can deduce as in \eqref{eq:mphi}
that
\begin{equation*}
	\|m_\vphi^n\|_{L^p(A_R)} \leq C\|\vr^n\|_{L^p(A_R)}^\frac{1}{2}\overset{n\rightarrow \infty}{\longrightarrow }0,
\end{equation*}
and hence it suffices to check that
\begin{equation*}
	m(t,x)= u(t,x)\vr_\vphi(t,x), \quad \text{whenever }\vr(t,x) \neq 0.
\end{equation*}
For this purpose, we consider the set
\begin{equation*}
	B^\epsilon_R = \Set{(t,x) \in B_R(0)\times (0,T); \vr(t,x) > \epsilon}.
\end{equation*}
From Egorov's theorem and the compactness of $\vr^n$ and $\widetilde \vr^n$ (Lemma \ref{lem:tilde}), we have the existence 
of a set $C_\eta \subset B^\epsilon_R$ with measure $|B_R^\epsilon \setminus C_\eta| < \eta$
on which $\widetilde \vr^n$ and $\vr^n$ converge uniformly to $\vr$. 
Then, for $n$ sufficiently large, 
\begin{equation*}
	\widetilde \vr^n \geq \epsilon/2 \quad \text{in } C_\eta,
\end{equation*}
and since
\begin{equation*}
	m_\vphi^n = \widetilde u^n \vr_\vphi^n = \frac{\widetilde j^n}{\widetilde \vr^n}\vr_\vphi^n,
\end{equation*}
we can pass to the limit on $C_\eta$ to deduce
\begin{equation*}
	m = \frac{j}{\vr}\vr_\vphi = u \vr_\vphi\quad \text{in } C_\eta.
\end{equation*}
Since this holds for all $\eta > 0$, we can conclude
\begin{equation*}
	m = u \vr_\vphi \quad \text{in }B_R^\epsilon,
\end{equation*}
for every $R$ and $\epsilon$. We conclude that, 
$$
m = u \vr_\vphi \quad \text{on }\Set{\vr > 0}.
$$

\end{proof}

\subsection*{Proof of Theorem \ref{thm:main}:}
The weak formulation of \eqref{eq:approx} reads
\begin{equation}\label{eq:apweak}
	\begin{split}
		&\int_0^T\int_{\R^{2d}}f^n(\psi_t + v\cdot \Grad_x \psi - \Grad_x \Psi\Grad_v \psi)~dvdxdt \\
		&\qquad := I^n_1 + I^n_2 - \int_{\R^2d}f_0^n\psi(0,\cdot)~dvdx, \quad \forall \psi \in C_c^\infty((0,T)\times \R^{2d}),
	\end{split}
\end{equation}
where we have introduced the quantities 
\begin{align*}
	I^n_1 &= -\int_0^T\int_{\R^{4d}}\Phi(x-y)f^n(x,v)f^n(y,w)(w-v)\Grad_v \psi(x,v)~dwdydvdxdt, \\
	I^n_2 &= - \int_0^T\int_{\R^{2d}}f^n(\widetilde u^n - v)\Grad_v \psi~dvdxdt.
\end{align*}
By virtue of \eqref{eq:weak}, we can pass to the limit in \eqref{eq:apweak} 
to conclude
\begin{equation}\label{eq:adone1}
	\begin{split}
		&\int_0^T\int_{\R^{2d}}f(\psi_t + v\cdot \Grad_x \psi - \Grad_x \Psi\Grad_v \psi)~dvdxdt \\
		&\qquad := I_1 + \lim_{n\rightarrow \infty}I_2^n-\int_{\R^2d}f_0\psi(0,\cdot)~dvdx,
	\end{split}
\end{equation}
where $I_1 = -\int_0^T\int_{\R^{4d}}\Phi(x-y)f(x,v)f(y,w)(w-v)\Grad_v \psi(x,v)~dwdydvdxdt$.

From Lemma \ref{lem:product}, we have that $f^n \widetilde u^n \weakstar fu$ 
in $L^\infty((0,T);L^q(\R^{2d}))$, for any $q < \frac{d+2}{d+1}$, 
and hence there is no problems with passing to the limit in $I^n_2$ 
to discover
\begin{equation*}
	\begin{split}
			\lim_{n\rightarrow \infty} I_2^n 
			&= - \lim_{n \rightarrow \infty}\int_{\R^{2d}}f^n(\widetilde u^n - v)\Grad_v \psi~dvdx 
			= - \int_{\R^{2d}}f(u - v)\Grad_v \psi~dvdx.
	\end{split}
\end{equation*}
By setting this in \eqref{eq:adone1} and recalling that $\Psi$ is arbitrary, we conclude that 
the limit $f$ is a weak solution to 
\begin{equation*}
	f_t + \Div_x(vf) - \Div_v (f \Grad_x \Psi) + \Div_v ( fF[f]) + \Div_v(f(u-v)) = 0.
\end{equation*} 
This concludes the proof of Theorem \ref{thm:main}.

\qed


\begin{thebibliography}{10}




 
%

\bibitem{CS-2007A} F.\ Cucker and S.\ Smale. Emergent behavior in flocks. {\em IEEE Transactions on automatic control}, {\bf 52} {\bf  no. 5}: 852-862,  2007.

\bibitem{CS-2007B} F.\ Cucker and S.\ Smale. On the mathematics of emergence. {\em Japanese Journal of Mathematics,} {\bf 2} {no. (1)}:197-227, 2007.


%
%
%
%
%
%
%
%
%

\bibitem{HT-2008} 
S.-Y. \ Ha, and E.\ Tadmor.
From particle to kinetic and hydrodynamic descriptions of flocking. 
{\em Kinet. Relat. Models} {\bf 1} {no. 3}:  415-435, 2008. 

\bibitem{KMT-2012-1}T.~Karper, A.~ Meller, and K.~Trivisa.
Existence of weak solutions to kinetic flocking models. Preprint 2012.

\bibitem{KMT-2012-2}T.~Karper, A.~ Meller, and K.~Trivisa.
Hydrodynamic limit of the kinetic Cucker-Smale flocking with strong local alignment. Preprint 2012.  


\bibitem{MT-2011} S.\ Motsch and E.\ Tadmor. A new model for self-organized dynamics and its flocking behavior. {\em Journal of Statistical Physics, Springer,} {\bf 141} {\bf (5)}: 923-947, 2011. 

%

\bibitem{PS-1998} 
B.\ Perthame and  P.E.\ Souganidis. A limiting case for velocity averaging. {\em Ann. Sci. \'{E}cole Norm. Sup.} {\bf (4)} {\bf 31} {\bf no. 4}: 591-598, 1998.

%


\end{thebibliography}
\end{document}